\documentclass{amsart}

\usepackage{amsmath,amssymb,amsfonts,enumerate,amsthm}

\newcommand{\CM}{\mbox{CM}\,}
\newcommand{\lk}{\mbox{link}\,}

\newtheorem{thm}{Theorem}[section]

\newtheorem{lem}[thm]{Lemma}
\newtheorem{prop}[thm]{Proposition}
\newtheorem{defn}[thm]{Definition}

\newtheorem{exam}[thm]{Example}
\newtheorem{rem}[thm]{Remark}

\numberwithin{equation}{section}

\begin{document}
 \bibliographystyle{amsplain}
\title{Cohen-Macaulay-ness in codimension for bipartite graphs}
\author{Hassan Haghighi}
\address{Hassan Haghighi\\Department of Mathematics, K. N. Toosi
     University of Technology, Tehran, Iran.}
\author{Siamak Yassemi}
\address{Siamak Yassemi\\School of Mathematics, Statistics \&
Computer Science, University of Tehran, and School of Mathematics,
Institute for Research in Fundamental Sciences (IPM), Tehran, Iran.}
\author{Rahim Zaare Nahandi}
\address{Rahim Zaare Nahandi\\School of Mathematics, Statistics \&
Computer Science, University of Tehran, Tehran, Iran.}
\thanks{Emails: haghighi@kntu.ac.ir, yassemi@ipm.ir,
rahimzn@ut.ac.ir} \keywords{Flag complex, Cohen-Macaulay complex in
codimension}

\subjclass[2000]{13H10, 05C75}

\begin{abstract}

\noindent Let $G$ be an unmixed bipartite graph of dimension $d-1$.
Assume that $K_{n,n}$, with $n\ge 2$, is a maximal complete
bipartite subgraph of $G$ of minimum dimension. Then $G$ is
Cohen-Macaulay in codimension $d-n+1$. This generalizes a
characterization of Cohen-Macaulay bipartite graphs by Herzog and
Hibi and a result of Cook and Nagel on unmixed Buchsbaum graphs.
Furthermore, we show that any unmixed bipartite graph $G$ which is
Cohen-Macaulay in codimension $t$, is obtained from a Cohen-Macaulay
graph by replacing certain edges of $G$ with complete bipartite
graphs. We provide some examples.
\end{abstract}
\maketitle

\section{Introduction}

%\begin{document}
\maketitle \noindent Cohen-Macaulay simplicial complexes are among
central research topics in combinatorial commutative algebra. While
characterization of such complexes is a far reaching problem, one
appeals to study specific families of Cohen-Macaulay simplicial
complexes. Flag complexes are among important families of complexes
recommended to study \cite[page 100]{S1995}. However, it is known
that a simplicial complex is Cohen-Macaulay if and only if its
barycentric subdivision is a Cohen-Macaulay flag complex. Therefore,
a characterization of Cohen-Macaulay flag complexes is equivalent to
a characterization of Cohen-Macaulay simplicial complexes.
Nevertheless, after all, the ideal of a flag complex is generated by
quadratic square-free monomials, which are simpler compared with
arbitrary square-free monomial ideals. Furthermore, it seems that,
expressing many combinatorial properties in terms of graphs are more
convenient. As some evidences, the characterization of unmixed
bipartite graphs by Villarreal \cite{V2007} and Cohen-Macaulay
bipartite graphs by Herzog and Hibi \cite{HH2005} are well expressed
in terms of graphs.

On the other hand, in the hierarchy of families of graphs with
respect to Cohen-Macaulay property, Buchsbaum complexes appear right
after Cohen-Macaulay ones. Unmixed bipartite Buchsbaum graphs were
characterized by Cook and Nagel \cite{CN2012} (also by the authors
\cite{HYZ2010}). Natural families of graphs in this hierarchy are
bipartite $\CM_t$ graphs, i.e., graphs that their independence
complexes are pure and Cohen-Macaulay in codimension $t$. The
concept of $\CM_t$ simplicial complexes were introduced in
\cite{HYZ2012} which is the pure version of simplicial complexes
Cohen-Macaulay in codimension $t$ studied by Miller, Novik and
Swartz \cite{MNS2010}. In this note, we give characterizations of
unmixed bipartite $\CM_t$ graphs in terms of its dimension and the
minimum dimension of its maximal nontrivial complete bipartite
subgraphs. Cook and Nagel showed that the only non-Cohen-Macaulay
unmixed bipartite graphs are complete bipartite graphs \cite[Theorem
4.10]{CN2012} and \cite[Theorem 1.3]{HYZ2010}. Our results are
generalizations of this fact to unmixed bipartite graphs which are
Cohen-Macaulay in arbitrary codimension. In the next section we
gather necessary definitions and known results to be used in the
rest of the paper. In Section 3 we improve some results on joins of
simplicial complexes and disjoint unions of graphs with respect to
the $\CM_t$ property. Section 4 is devoted to two characterizations
of bipartite $\CM_t$ graphs and some examples.

\section{Preliminaries}
For basic definitions and general facts on simplicial complexes we
refer to the book of Stanley \cite{S1995}. By a complex we will
always mean a simplicial complex. Let $G = (V, E)$ be a simple graph
with vertex set $V$ and edge set $E$. The {\it inclusive
neighborhood} of $v\in V$ is the set $N[v]$ consisting of $v$ and
vertices adjacent to $v$ in $G$. The {\it independence complex} of
$G = (V, E)$ is the complex Ind$(G)$ with vertex set $V$ and with
faces consisting of independent sets of vertices of $G$, i.e., sets
of vertices of $G$ where no two elements of them are adjacent. These
complexes are called {\it flag complexes}, and their Stanley-Reisner
ideal is generated by quadratic square-free monomials. By {\it
dimension} of a graph $G$ we mean the dimension of the complex
Ind$(G)$. A graph $G$ is said
to be {\it unmixed} if Ind$(G)$ is pure.\\
For an integer $t \ge 0$, a complex $\Delta$ is called $\CM_t$ if it
is pure and for every face $F\in\Delta$ with $\#(F) \ge t$,
$\lk_{\Delta}(F)$ is Cohen-Macaulay. This is the same as pure
complexes which are Cohen-Macaulay in codimension $t$. Accordingly,
$\CM_0$ and $\CM_1$ complexes are precisely Cohen-Macaulay and
Buchsbaum complexes, respectively. Clearly, a $\CM_t$ complex is
$\CM_r$ for all $r\ge t$ and a complex of dimension $d-1$ is always
$\CM_{d-1}$. One uses the convention that for $t<0$, ${\rm CM}_t$
would mean ${\rm CM}_0$. A graph $G$ is called $\CM_t$ if Ind$(G)$
is $\CM_t$. A basic tool for checking $\CM_t$ property of complexes
is the following lemma.

\begin{lem}\label{basiclemma}(\cite[Lemma 2.3]{HYZ2012}) Let $t\ge 1$ and let $\Delta$ be a
nonempty complex. Then the following are equivalent:
\begin{itemize}

\item[(i)] $\Delta$ is a ${\rm CM}_t$ complex.

\item[(ii)] $\Delta$ is pure and $\lk_\Delta(v)$ is ${\rm CM}_{t-1}$ for
every vertex $v\in\Delta$.

\end{itemize}

\end{lem}

By the straightforward identity $\lk_{{\rm Ind}(G)}(v) = {\rm
Ind}(G\setminus N[v])$, the counter-part of this lemma for graphs
would be the following:

\begin{lem}\label{lemmagraphs} Let $t\ge 1$ and let $G$ be a
graph. Then the following are equivalent:
\begin{itemize}

\item[(i)] $G$ is a ${\rm CM}_t$ graph.

\item[(ii)] $G$ is unmixed and $G\setminus N[v]$ is a ${\rm CM}_{t-1}$ graph for
every vertex $v \in G$.

\end{itemize}

\end{lem}

We recall some basic relevant facts on bipartite graphs. A graph $G
= (V,E)$ is called {\it bipartite} if $V$ is a disjoint union of a
partition $V_1$ and $V_2$ and $E \subset V_1 \times V_2$. If
$\#(V_1) = m$ and $\#(V_2) = n$ and $E = V_1 \times V_2$, then $G$
is the {\it complete} bipartite graph $K_{m,n}$. We will be
interested in unmixed complete bipartite graphs $K_{n,n}$.\\

Unmixed bipartite graphs are characterized by Villarreal in the
following result.

\begin{thm}\label{Villarreal} \cite[Theorem 1.1]{V2007} Let $G$ be a bipartite graph without
isolated vertex. Then $G$ is unmixed if and only if there is a
partition $V_1 = \{x_1, \cdots, x_n\}$ and $V_2 = \{y_1, \cdots,
y_n\}$ of vertices of $G$ such that
\begin{itemize}

\item[(1)] $x_iy_i$ is an edge in $G$ for $1\le i \le n$ and

\item[(2)] If $x_iy_j$ and $x_jy_k$ are edges in $G$, for some distinct $i$,
$j$ and $k$, then $x_iy_k$ is an edge in $G$.

\end{itemize}
\end{thm}

In this case, such a partition and ordering is called a {\it pure
order} of $G$. The edges $x_iy_i$, $i =1,\cdots, n$ are called a
{\it perfect matching} edges of $G$. A pure order is said to have a
{\it cross} if, for some $i\ne j$, $x_iy_j$ and $x_jy_i$ are both
edges in $G$. Otherwise, the order is called {\it cross-free} (see
\cite[$\S$ 4]{CN2012}). For unmixed bipartite graphs, being
cross-free is independent of an ordering of vertices of $G$. More
precisely, if $G$ has a cross in some pure ordering, it has a cross
in every pure
ordering \cite[Lemma 4.5]{CN2012}.\\

An immediate consequence of Theorem \ref{Villarreal} is the
following useful lemma.

\begin{lem}\label{compneighb} Let $G$ be an unmixed bipartite
graph with pure order of vertices $(\{x_1,\cdots, x_d\},
\{y_1,\cdots, y_d\})$ and let $K_{n,n}$ be a complete bipartite
subgraph of $G$ on $(\{x_{i_1},\cdots, x_{i_n}\}, \{y_{i_1},\cdots,
y_{i_n}\})$.
\begin{itemize}

\item[(i)] If $x_jy_{i_k}$ is an edge in $G$ for some $j$ and $k$, then $x_jy_{i_l}$
is an edge in $G$ for all $l=1,\cdots, n$.
\item[(ii)] If $x_{i_k}y_{j}$ is an edge in $G$ for some $k$ and $j$, then $x_{i_l}y_j$
is an edge in $G$ for all $l=1,\cdots, n$.

\end{itemize}
\end{lem}
\begin{proof}
The assertion $(i)$ is immediate by Theorem \ref{Villarreal} because
$x_{i_k}y_{i_l}$ is an edge in $K_{n,n} \subset G$ for all
$l=1,\cdots, n$. Also $(ii)$ follows because $x_{i_l}y_{i_k}$ is an
edge in $K_{n,n} \subset G$ for all $l=1,\cdots, n$.
\end{proof}

There are also at least two nice characterization of Cohen-Macaulay
bipartite graphs.

\begin{thm} \label{HerHib} \cite[Theorem 3.4]{HH2005} Let $G$ be a bipartite graph
without isolated vertices. Then $G$ is Cohen-Macaulay if and only if
there is a pure ordering $V_1 = \{x_1, \cdots, x_n\}$ and $V_2 =
\{y_1, \cdots, y_n\}$ of vertices of $G$ such that $x_iy_j$ being in
$G$ implies $i \le j$.
\end{thm}

The ordering in Theorem \ref{HerHib} is called a {\it Macaulay
order} of vertices of $G$.

\begin{prop} \label{cookNag} \cite[Proposition 4.8]{CN2012}  Let $G$ be a bipartite graph.
Then $G$ is Cohen-Macaulay if and only if $G$ has a cross-free pure
order.
\end{prop}

Bipartite Buchsbaum graphs are also classified. First recall that a
complex is Buchsbaum if and only if it is pure and the link of each
vertex is Cohen-Macaulay \cite{Sch1981}. Thus, a graph is Buchsbaum
if and only if $G$ is unmixed and for each vertex $v\in G$,
$G\setminus N[v]$ is Cohen-Macaulay. For bipartite graphs there is a
sharper result. Complete bipartite graphs are well-known to be
Buchsbaum (e.g., see \cite[Proposition 2.3]{ZZ2004}). But indeed,
the converse is also true.

\begin{thm} \label{HYZBuch} (see \cite[Theorem 4.10]{CN2012} or
\cite[Theorem1.3]{HYZ2010}) Let $G$ be a bipartite graph. Then $G$
is Buchsbaum if and only if $G$ is a complete bipartite graph
$K_{n,n}$ for some $n\ge 2$, or $G$ is Cohen-Macaulay.
\end{thm}

\section{Joins of $\CM_t$ complexes and disjoint unions of $\CM_t$ graphs}

It is known that the join of two complexes is Cohen-Macaulay if and
only if they are both Cohen-Macaulay (see \cite{S1985} and
\cite{F1988}). If $\Delta$ is a $\CM_r$ complex of dimension $d-1$
and $\Delta'$ is a $\CM_{r'}$ complex of dimension $d'-1$, then
their join $\Delta \ast \Delta'$ is a $\CM_t$ complex where $t={\rm
max}\{d+r', d'+r\}$ \cite[Proposition 2.10]{HYZ2012}. However, if
one of the complexes is Cohen-Macaulay, this result could be
strengthened. Below we combine this with relevant known results.

\begin{thm}\label{join} Let $\Delta$ and $\Delta'$ be two
complexes of dimensions $d-1$ and $d'-1$, respectively. Then
\begin{itemize}

\item[(i)] The join complex $\Delta \ast \Delta'$ is Cohen-Macaulay
if and only if both $\Delta$ and $\Delta'$ are so.

\item[(ii)] If $\Delta$ is Cohen-Macaulay and $\Delta'$ is
$\CM_{r'}$ for some $r'\ge 1$, then $\Delta \ast \Delta'$ is
$\CM_{d+r'}$ (independent of $d'$). This is sharp, i.e., if
$\Delta'$ is not $\CM_{r'-1}$, then $\Delta \ast \Delta'$ is not
$\CM_{d+r'-1}$. In particular, a cone on $\Delta'$ is $\CM_{r'+1}$.

\item[(iii)] If $\Delta$ is $\CM_r$ and $\Delta'$ is $\CM_{r'}$ for
some $r, r' \ge 1$, then $\Delta \ast \Delta'$ is $\CM_t$ where
$t={\rm max}\{d+r', d'+r\}$. Conversely, if $\Delta \ast \Delta'$ is
$\CM_t$, then $\Delta$ is $\CM_{t-d'}$ and $\Delta'$ is $\CM_{t-d}$.
\end{itemize}

\end{thm}

\begin{proof}
The statement in $(i)$ is proved by Sava \cite{S1985} and Fr\"oberg
\cite{F1988}. The assertion $(iii)$ is proved in \cite[Theorem
2.10]{HYZ2012}. We prove $(ii)$ using induction on $d+r' \ge 2$. Let
$d+r'=2$, i.e., $d=1$ and $r'=1$.  Then $\Delta = \{v\}$ is a
singleton. Thus $\lk_{\Delta \ast \Delta'}(v) = \Delta'$, which is
$\CM_1$. For $v \in \Delta'$, $\lk_{\Delta \ast \Delta'}(v) = \Delta
\ast \lk_{\Delta'}(v)$, which is Cohen-Macaulay by $(i)$. Thus by
Lemma \ref{basiclemma}, $\Delta \ast \Delta'$ is $\CM_2$. Now let
$d+r'\ge 2$. Let $v\in \Delta$. Then, $\lk_{\Delta \ast \Delta'}(v)
= \lk_\Delta (v)\ast \Delta'$. But $\lk_\Delta (v)$ is
Cohen-Macaulay of dimension less than $d-1$, and $\Delta'$ is
$\CM_{r'}$. Thus by induction hypothesis $\lk_{\Delta\ast \Delta'}
(v)$ is $\CM_{d-1+r'}$. If $v\in \Delta'$, then $\lk_{\Delta \ast
\Delta'}(v) = \Delta' \ast \lk_{\Delta'}(v)$. But $\lk_{\Delta'}(v)$
is $\CM_{r'-1}$ and hence $\lk_{\Delta\ast \Delta'} (v)$ is again
$\CM_{d+r'-1}$. Therefore, $\Delta \ast \Delta'$ is $\CM_{d+r'}$. To
prove that this result is sharp, proceed by induction on $d\ge 1$.
Indeed, in this case, for any $v\in \Delta$,  $\lk_\Delta(v)$ has
dimension less than $d-1$ and hence by induction hypothesis,
$\lk_{\Delta \ast \Delta'}(v) = \lk_\Delta(v) \ast \Delta$ is not
$\CM_{d+r'-2}$. Therefore, $\Delta \ast \Delta'$ is not
$\CM_{d+r'-1}$
\end{proof}

Let $G\sqcup G'$ denote the disjoint union of graphs $G$ and $G'$.
By the fact that Ind$(G\sqcup G') =$ Ind$(G)\ast$Ind$(G')$, the
counter-part of Theorem \ref{join} for graphs will be the following.

\begin{thm}\label{joingraphs} Let $G$ and $G'$ be two
graphs on disjoint sets of vertices and of dimensions $d-1$ and
$d'-1$, respectively. Then
\begin{itemize}

\item[(i)] The graph $G\sqcup G'$ is Cohen-Macaulay
if and only if both $G$ and $G'$ are so.

\item[(ii)] If $G$ is Cohen-Macaulay and $G'$ is
$\CM_{r'}$ for some $r'\ge 1$, then $G\sqcup G'$ is $\CM_{d+r'}$. If
$G'$ is not $\CM_{r'-1}$, then $G\sqcup G'$ is not $\CM_{d+r'-1}$.

\item[(iii)] If $G$ is $\CM_r$ and $G'$ is $\CM_{r'}$ for
some $r, r' \ge 1$, then $G\sqcup G'$ is $\CM_t$ where $t={\rm
max}\{d+r', d'+r\}$. Conversely, if $G\sqcup G'$ is $\CM_t$, then
$G$ is $\CM_{t-d'}$ and $G'$ is $\CM_{t-d}$.
\end{itemize}

\end{thm}

\section{Two characterizations of bipartite $\CM_t$ graphs}

We now restrict to the case of bipartite graphs. Since
Cohen-Macaulay bipartite graphs are characterized by Herzog and Hibi
\cite[Theorem 3.4]{HH2005}, and also in a different version by Cook
and Nagel \cite[Proposition 4.8]{CN2012}, we consider the
non-Cohen-Macaulay case.

\begin{thm}\label{main} Let $G$ be an unmixed bipartite
graph of dimensions $d-1$. Let $K_{n,n}$, with $n\ge 2$, be a
maximal complete bipartite subgraph of $G$ of minimum dimension.
Then $G$ is $\CM_{d-n+1}$ but it is not $\CM_{d-n}$.
\end{thm}

\begin{proof}
We prove both assertions by induction on $d \ge 2$. If $d = 2$ then
$G=K_{2,2}$ which is $\CM_1$ but it is not Cohen-Macaulay. Assume
that $d
> 2$. We show that for every $v\in G$, $G\setminus N[v]$ is
$\CM_{n-d}$ and for some $v\in G$ it is not $\CM_{n-d-1}$. Let
$(\{x_1, \cdots, x_d\}, \{y_1, \cdots, y_d\})$ be a pure order of
$G$. Let $x_i$ be a vertex of some maximal bipartite subgraph
$K_{m,m}$ with $m\ge n$. Then $G\setminus N[x_i]$ is a disjoint
union of $c\ge m-1$ isolated vertices and an unmixed bipartite graph
$H$ of dimension $d-c-2$. The graph $H$ is unmixed because
Ind$(G\setminus N[x_i]) = \lk_{x_i}({\rm Ind}(G))$, and any link of
a pure complex is pure. But $G\setminus N[x_i] = \{x_{i_1}, \cdots,
x_{i_c}\} \sqcup H$ is unmixed if and only if $H$ is so. Observe
that if $y_{j_0}$ is a vertex of a maximal bipartite subgraph of $G$
and $y_{j_0}\in N[x_i]$, then by Lemma \ref{compneighb}, all $y_j$
vertices of this subgraph belong to $N[x_i]$. Thus if $H$ has no
crosses, by Proposition \ref{cookNag} it is Cohen-Macaulay.
Otherwise, the minimum dimension of maximal complete bipartite
subgraphs of $H$ will not be less than the minimum dimension of such
subgraphs in $G$. Hence by the induction hypothesis $H$ is
$\CM_{d-c-n}$ and by Theorem \ref{joingraphs}(ii), $G\setminus
N[x_i]$ is $\CM_{n-d}$. If $x_i$ does not belong to any maximal
bipartite subgraph of $G$ of positive dimension, then $G\setminus
N[x_i]$ is a disjoint union of $c\ge 0$ isolated vertices and an
unmixed bipartite graph $H$ of dimension $d-c-2$. Hence $H$ is
$\CM_{d-c-n}$ and by Theorem \ref{joingraphs}(ii), $G\setminus
N[x_i]$ is $\CM_{d-n}$. A similar argument reveals that for any
$y_i\in G$, the graph $G\setminus N[y_i]$ is $\CM_{d-n}$. Therefore,
by Lemma \ref{lemmagraphs}, $G$ is $\CM_{d-n+1}$. We now proceed the
induction step to show that this result is sharp. Let $d>2$ and let
$K_{n,n}$, $n\ge 2$, be a maximal bipartite subgraph of $G$ of
minimum dimension. Take $x_i \in G\setminus K_{n,n}$. First assume
that $x_i$ is not adjacent to any vertex in $K_{n,n}$ and consider
$G\setminus N[x_i]$. Let $G\setminus N[x_i]$ be the disjoint union
of $c\ge 0$ isolated vertices and an unmixed bipartite graph $H$ of
dimension $d-c-2$. Then $H$ contains $K_{n,n}$ and hence by
induction hypothesis $H$ is sharp $\CM_{d-c-n}$ and $G\setminus
N[x_i]$ is sharp  $\CM_{d-n}$. Therefore, $G$ can not be
$\CM_{d-n}$. Now assume that $x_iy_j \in G$ for some $j$ with
$y_j\in K_{n,n}$. Then by purity of the order, all $y_k \in K_{n,n}$
is adjacent to $x_i$. But then $y_i$ is not adjacent to any vertex
of $K_{n,n}$, because otherwise, $ K_{n,n}$ will not be maximal. In
this case, consider $G\setminus N[y_i]$ and proceed similar to the
previous case.
\end{proof}

As a second characterization of bipartite $\CM_t$ graphs, we show
that any $\CM_t$ graph is obtained from a Cohen-Macaulay graph $H$
by replacing the perfect matching edges of $H$ by complete bipartite
graphs. This statement will be more precise in the next theorem. But
first we provide a definition and a lemma.

\begin{defn}\label{replacement} Let $H$ be an unmixed bipartite graph with
pure order $$(\{x_1, \cdots, x_r\},\{y_1, \cdots, y_r\}).$$ For a
fixed $i$, by {\it replacing} the edge $x_iy_i\in H$ with a complete
bipartite graph
$$K_{n_i,n_i} = \{x_{i1}, \cdots, x_{in_i}\}\times\{y_{i1}, \cdots, y_{in_i}\}$$
we mean a bipartite graph $H'$ with vertex set $$(\{x_1, \cdots,
x_{i-1}, x_{i1}, \cdots, x_{in_i}, x_{i+1}, \cdots, x_r\}, \{y_1,
\cdots, y_{i-1}, y_{i1}, \cdots, y_{in_i}, y_{i+1}, \cdots,
y_r\}),$$ preserving all adjacencies, i.e.,
\begin{itemize}
\item[(i)] $x_sy_t\in H'$ for all $s, t \ne i$ if and only if
$x_ty_s\in H$, \item[(ii)] $x_{i_k}y_j\in H'$ for all $k$ if and
only if $x_iy_j\in H$, \item[(iii)] $x_jy_{i_k}\in H'$ for all $k$
if and only if $x_jy_i\in H$.
\end{itemize}
\end{defn}

\begin{lem}\label{unmixedreplace} Let $G$ be an unmixed bipartite
graph with pure order on the vertex set $V(G) = V\cup W$ where $V =
\{x_1, \cdots, x_d\}$ and $W = \{y_1, \cdots, y_d\}$. Let
$n_1,\cdots, n_d$ be any positive integers. Let $G' = G(n_1,\cdots,
n_d)$ be the graph obtained by replacing each edge $x_iy_i$ with the
complete bipartite graph $K_{n_i,n_i} = \{x_{i1}, \cdots,
x_{in_i}\}\times \{y_{i1}, \cdots, y_{in_i}\}$ for all $i = 1,\dots,
d$. Then $G'$ is also unmixed.
\end{lem}
\begin{proof} Let $K_{n_i,n_i} = \{x_{i1}, \cdots, x_{in_i}\}
\times \{y_{i1}, \cdots, y_{in_i}\}$. Then $$V(G') = (\{x_{11},
\cdots, x_{1n_1}, \cdots, x_{d1}, \cdots, x_{dn_d}\}, \{y_{11},
\cdots, y_{1n_1}, \cdots, y_{d1}, \cdots, y_{dn_d}\})$$ is a pure
order of $G'$. In fact, for all $i, r$, $x_{ir}y_{ir} \in G'$. Also
if $x_{ir}y_{js} \in G'$ and $x_{js}y_{kt} \in G'$, then $x_iy_j \in
G$ and $x_jy_k \in G$, and hence, $x_iy_k \in G$. Thus by the
construction of $G'$, $x_{ir}y_{kt} \in G'$.
\end{proof}

\begin{thm} \label{macaulayreplace} Let $G$ be a Cohen-Macaulay bipartite
graph with a Macaulay order on the vertex set $V(G) = V\cup W$ where
$V = \{x_1, \cdots, x_d\}$ and $W = \{y_1, \cdots, y_d\}$. Let $n_1,
\cdots, n_d$ be any positive integers with $n_i \ge 2$ for at list
one $i$. Let $G' = G(n_1, \cdots, n_d)$ be the graph obtained by
replacing each edge $x_iy_i$ with the complete bipartite graph
$K_{n_i,n_i}$ for all $i = 1,\dots, d$. Let $n_{i_0} = {\rm min}
\{n_i > 1: i = 1, \cdots, d\}$, $n= \sum_{i=1}^{d}n_i$. Then $G'$ is
exclusively a $\CM_{n-n_{i_0}+1}$ graph. Furthermore, any bipartite
$\CM_t$ graph is obtained by such a replacement of complete
bipartite graphs in a unique bipartite Cohen-Macaulay graph.
\end{thm}

\begin{proof} The first claim follows by Lemma \ref{unmixedreplace}
and Theorem \ref{main}. We settle the second claim. Let $G$ be a
bipartite $\CM_t$ graph with a pure order of vertices. Let
$K_{n_1,n_1}, \cdots, K_{n_d,n_d}$ be the maximal bipartite
subgraphs of $G$, where $n_i \ge 1$ for all $i$. Observe that, by
maximality, these complete subgraphs of $G$ are disjoint. Choose one
edge $x_{i1}y_{i1}$ from each subgraph $K_{n_i,n_i}$ for all $i = 1,
\cdots, d$. Let $H$ be the induced subgraph of $G$ on the vertex set
$(\{x_{11}, \cdots, x_{d1}\},\{y_{11}, \cdots, y_{d1}\})$. By Lemma
\ref{compneighb}, $H$ is independent of the choice of particular
edge $x_{i1}y_{i1}$ from $K_{n_i,n_i}$ and hence $H$ is unique.
Since the ordering of vertices of $G$ is a pure order, its
restriction to $H$ is also pure. Thus, $H$ is an unmixed bipartite
graph. But by the maximality of the complete bipartite subgraphs
$K_{n_i,n_i}$, and the construction of $H$, it is cross-free.
Therefore, by Proposition \ref{cookNag}, $H$ is Cohen-Macaulay. Now
any edge $x_{i1}y_{i1}$ replace in $H$ with $K_{n_i,n_i}$ for all $i
= 1, \cdots, d$, preserving all other adjacencies. Let $H'$ be the
resulting graph. Then by the construction, $G = H'$, as required.
\end{proof}

\begin{rem}\label{remark} Let $H$ be a bipartite Cohen-Macaulay graph and let $G
= H'$ be a bipartite $\CM_t$ graph obtained from $H$ by the
replacing process described above. Assume that $G$ is is not
$\CM_{t-1}$ and $t \ge 2$. Using the the results of this section,
the following observations are immediate.

First of all, $1 \le {\rm dim} H \le t-1$. Because if ${\rm dim} H
\ge t$ and we replace just one $K_{n,n}$ with $n\ge 2$, then $G$ is
strictly $\CM_r$ with $r\ge t+1$. On the other hand, if ${\rm dim} H
= 0$, then $G$ is $\CM_1$.

If ${\rm dim} H = t-1$, then only one $K_{n,n}$ with $n\ge 2$ can be
replaced. Because replacing at least two $K_{n,n}$ with $n\ge 2$,
$G$ is strictly $\CM_r$ with $r \ge t+1$.

If ${\rm dim} H = t-1$, for replacing just one $K_{n,n}$, $n$ is
arbitrary and hence $G$ is of dimension $n+t-2$.

If ${\rm dim} H \le t-2$, the number of replacements should be at
least $2$. Again because if with one replacement of $K_{n,n}$, $n
\ge 2$, $G$ would be $\CM_r$ with $r \le t-1$.

When ${\rm dim} H \le t-2$, the maximum number of replacements of
$K_{n,n}$, $n \ge 2$, is at most $t - {\rm dim} H$ which may occur
replacing $K_{2,2}$'s.

For ${\rm dim} H \le t-2$, the maximum size of $K_{n,n}$ to be
replaced is also $n = t - {\rm dim} H$ which may occur when we have
two replacements.\\
\end{rem}

Using these remarks we may easily distinguish all bipartite $\CM_t$
graphs for $t = 2, 3, 4$.\\

\begin{exam}\label{exam1}
Bipartite $\CM_2$ graphs which are not Buchsbaum. Using the notation
of Remark \ref{remark} we have ${\rm dim} H = 1$. There are just two
non-isomorphic bipartite Cohen-Macaulay graphs of dimension one. By
replacing process, they produce two types of bipartite $\CM_2$
graphs which are not Buchsbaum. They are of arbitrary dimensions.
More precisely, one such graph is the disjoint union of an edge
$x_1y_1$ with $K_{n_2,n_2} = \{x_{21},\cdots,
x_{2n_2}\}\times\{y_{21},\cdots, y_{2n_2}\}$, $n_2\ge 2$, and the
other one consists of the first graph together with the edges
$x_1y_{2i}$ for all $i = 1,\cdots, n_2$. The second graph with $n_2
= 3$ could be depicted in Figure 1.
\end{exam}
\begin{picture}(100,60)(-100,30)
%% Lines
\put(0,0){\line(0,1){50}} \put(0,50){\line(1,-1){50}}
\put(50,50){\line(1,-1){50}} \put(50,50){\line(2,-1){100}}
\put(100,50){\line(1,-1){50}} \put(50,0){\line(0,1){50}}
\put(100,0){\line(0,1){50}} \put(150,0){\line(0,1){50}}
\put(100,0){\line(1,1){50}} \put(50,0){\line(2,1){100}}
\put(50,0){\line(1,1){50}}
\put(0,50){\line(2,-1){100}}\put(0,50){\line(3,-1){150}}
%% Circle
\put(0,0){\circle*{5}} \put(0,50){\circle*{5}}
\put(50,0){\circle*{5}} \put(50,50){\circle*{5}}
\put(100,0){\circle*{5}} \put(100,50){\circle*{5}}
\put(150,0){\circle*{5}} \put(150,50){\circle*{5}}
%% Labels
\put(-4,55){$x_1$} \put(46,55){$x_{21}$} \put(96,55){$x_{22}$}
\put(146,55){$x_{23}$}

\put(-4,-10){$y_1$} \put(46,-10){$y_{21}$} \put(96,-10){$y_{22}$}
\put(146,-10){$y_{23}$}
\end{picture}
\vspace{2cm}
$$Figure \ 1$$

\begin{exam}
Bipartite $\CM_3$ graphs which are not $\CM_2$. For these graphs
${\rm dim} H = 1, 2$.\\

If ${\rm dim} H = 1$, by Example \ref{exam1}, there are just two
bipartite $\CM_3$ graphs by replacing two edges of a perfect
matching by $K_{2,2}$'s. In this case, ${\rm dim} G = 3$. (see
Figure 2, and Figure 3).\\

If ${\rm dim} H = 2$, then there are $4$ non-isomorphic bipartite
Cohen-Macaulay graphs of dimension $2$. By replacing one perfect
matching edge with $K_{n,n}$ of arbitrary size in each
Cohen-Macaulay graph, they produce $7$ types of bipartite $\CM_3$
graphs which are not $\CM_2$. Note that depending on the choice of
the edge to be replaced in each case, we may get non-isomorphic
bipartite graphs. In this case ${\rm dim} G = n + 1$.
\end{exam}
\newpage
\begin{picture}(100,60)(-100,30)
%% Lines
\put(0,0){\line(0,1){50}} \put(0,50){\line(1,-1){50}}
%\put(50,50){\line(2,-1){100}}
\put(100,50){\line(1,-1){50}} \put(50,0){\line(0,1){50}}
\put(0,0){\line(1,1){50}} \put(100,0){\line(0,1){50}}
\put(150,0){\line(0,1){50}}
\put(100,0){\line(1,1){50}} %\put(0,50){\line(2,-1){100}}
%\put(0,50){\line(3,-1){150}}%
%% Circle
\put(0,0){\circle*{5}} \put(0,50){\circle*{5}}
\put(50,0){\circle*{5}} \put(50,50){\circle*{5}}
\put(100,0){\circle*{5}} \put(100,50){\circle*{5}}
\put(150,0){\circle*{5}} \put(150,50){\circle*{5}}
%% Labels
\put(-4,55){$x_{11}$} \put(46,55){$x_{12}$} \put(96,55){$x_{21}$}
\put(146,55){$x_{22}$}

\put(-4,-10){$y_{11}$} \put(46,-10){$y_{12}$} \put(96,-10){$y_{21}$}
\put(146,-10){$y_{22}$}
\end{picture}
\vskip 2cm
$$Figure \ 2$$
\begin{picture}(100,60)(-100,30)
%% Lines
\put(0,0){\line(0,1){50}} \put(0,50){\line(1,-1){50}}
\put(50,50){\line(2,-1){100}} \put(100,50){\line(1,-1){50}}
\put(50,0){\line(0,1){50}}\put(0,0){\line(1,1){50}}
\put(100,0){\line(0,1){50}} \put(150,0){\line(0,1){50}}
\put(100,0){\line(1,1){50}} \put(0,50){\line(2,-1){100}}
\put(0,50){\line(3,-1){150}}
%% Circle
\put(50,50){\line(1,-1){50}}
 \put(0,0){\circle*{5}} \put(0,50){\circle*{5}}
\put(50,0){\circle*{5}} \put(50,50){\circle*{5}}
\put(100,0){\circle*{5}} \put(100,50){\circle*{5}}
\put(150,0){\circle*{5}} \put(150,50){\circle*{5}}
%% Labels
\put(-4,55){$x_{11}$} \put(46,55){$x_{12}$} \put(96,55){$x_{21}$}
\put(146,55){$x_{22}$}

\put(-4,-10){$y_{11}$} \put(46,-10){$y_{12}$} \put(96,-10){$y_{21}$}
\put(146,-10){$y_{22}$}
\end{picture}
\vskip 2cm
$$Figure \ 3$$

\begin{exam}
Bipartite $\CM_4$ graphs which are not $\CM_3$. For these graphs
${\rm dim} H = 1, 2, 3$.\\

If ${\rm dim} H = 1$, there are two bipartite $\CM_4$ graphs
obtained by replacing two edges of a perfect matching by
$K_{3,3}$'s. In this case, ${\rm dim} G = 5$. And, similarly, there
are two others obtained by replacing one edge with $K_{2,2}$ and
another edge with $K_{3,3}$.
In this case, ${\rm dim} G = 5$.\\

If ${\rm dim} H = 2$, then while there are $4$ non-isomorphic
bipartite Cohen-Macaulay graphs of dimension $2$, by replacing two
perfect matching edges with $K_{2,2}$'s in each Cohen-Macaulay
graph, they produce $7$ bipartite $\CM_4$ graphs which are not
$\CM_3$. They all have dimension $4$.\\

If ${\rm dim} H = 3$, then there are $10$ non-isomorphic bipartite
Cohen-Macaulay graphs of dimension $3$. Replacing one perfect
matching edge with $K_{n,n}$, $n\ge 2$, in each Cohen-Macaulay
graph, they produce $25$ bipartite $\CM_4$ graphs which are not
$\CM_3$. They all have dimension $n+2$. Out of all 36 bipartite $\CM_4$ graphs,
21 graphs are connected.\\
\end{exam}

\section*{Acknowledgments}
The authors would like to thank J. Herzog and U. Nagel for some
fruitful discussions on the subject of this paper. This work has
been supported by Center for International Studies \& Collaboration
(CISSC) and French Embassy in Tehran in the framework of the
Gundishapur project 27462PL on the Homological and Combinatorial
Aspects of Commutative Algebra and Algebraic Geometry. The research
of R. Zaare-Nahandi has been partially supported by research grant
no. of University of Tehran. The research of S. Yassemi was in part
supported by a grant from IPM (No. 91130214).

\bibliographystyle{amsplain}

\end{document}